\documentclass[12pt]{amsart}
	\usepackage[utf8]{inputenc}
	\usepackage{a4wide}
	\usepackage{amsfonts}
	\usepackage{amsmath}
	\usepackage{amssymb}
	\usepackage{amsthm}
	\usepackage{mathrsfs}
	\usepackage{enumerate}
	\usepackage[all]{xy}
	\usepackage[colorlinks,citecolor=blue]{hyperref}
	\usepackage{color}
	\usepackage{graphicx}
	
	\newcommand{\PP}{\mathbb{P}}
	
	\renewcommand{\AA}{\mathbb{A}}
	\newcommand{\ZZ}{\mathbb{Z}}	
	\newcommand{\NN}{\mathbb{N}}	
	\newcommand{\QQ}{\mathbb{Q}}

	\newcommand{\Hh}{\mathscr{H}}

	\newcommand{\Mm}{\mathscr{M}}	
		
	\newcommand{\pp}{\wp}	
	\newcommand{\Ff}{\mathscr{F}}	
	\newcommand{\Oo}{\mathscr{O}}	
		
	\newcommand{\CC}{\mathbb{C}}	

	\renewcommand{\bar}{\overline}
	\newcommand{\et}{{\acute{e}t}}
	\newcommand{\Et}{{\acute{E}t}}

	\DeclareMathOperator{\chara}{char}

	\DeclareMathOperator{\Spec}{Spec}
	
	\DeclareMathOperator{\Proj}{Proj}

	\DeclareMathOperator{\ass}{ass}

	\newcommand{\isom}{\simeq}
	\newcommand{\xto}{\xrightarrow}

	\renewcommand{\phi}{\varphi}
	\renewcommand{\emptyset}{\varnothing}
	\renewcommand{\epsilon}{\varepsilon}
	\renewcommand{\tilde}{\widetilde}
	\renewcommand{\to}{\longrightarrow}
	
	\newcommand{\indown}{\mathrel{\rotatebox[origin=c]{-90}{$\in$}}}

	\newcommand{\stacks}[1]{\cite[\href{http://stacks.math.columbia.edu/tag/#1}{Tag #1}]{stacks-project}}

	\theoremstyle{plain}
		\newtheorem{lemma}[subsection]{Lemma}
		\newtheorem{corollary}[subsection]{Corollary}
		\newtheorem{proposition}[subsection]{Proposition}
		\newtheorem{theorem}[subsection]{Theorem}
		\newtheorem*{theorem*}{Theorem}
	\theoremstyle{definition}
		\newtheorem{definition}[subsection]{Definition}

	\numberwithin{equation}{subsection}

\begin{document}

\begin{abstract}
	A technical ingredient in Faltings' original approach to $p$-adic comparison theorems involves the construction of $K(\pi, 1)$-neighborhoods for a smooth scheme $X$ over a mixed characteristic dvr with a perfect residue field: every point $x\in X$ has an open neighborhood $U$ whose general fiber is a $K(\pi, 1)$ scheme (a notion analogous to having a contractible universal cover). We show how to extend this result to the logarithmically smooth case, which might help to simplify some proofs in $p$-adic Hodge theory. The main ingredient of the proof is a variant of a trick of Nagata used in his proof of the Noether Normalization Lemma.	
\end{abstract}

\date{\today}
\author{Piotr Achinger}
\title{$K(\pi, 1)$-neighborhoods and comparison theorems}

\maketitle

\setcounter{tocdepth}{1}
\tableofcontents

\section{Introduction}
\label{s:intro}

\subsection{} 
\label{intro:summary}
This paper contains several results about the \'etale topology of schemes over discrete valuation rings, mostly with applications to $p$-adic Hodge theory in mind. We prove the existence of $K(\pi, 1)$-neighborhoods for a log smooth scheme over a mixed characteristic discrete valuation ring (Theorem \ref{main}, see \ref{intro:main} below), and a comparison theorem between \'etale cohomology and the cohomology of Faltings' topos (Corollary \ref{comparisoncor}, see \ref{intro:comparison}). These could be used to simplify some arguments Faltings' second paper on $p$-adic Hodge theory \cite{FaltingsAlmost} using the approach of the first one \cite{Faltings}. A reader familiar with the notion of a $K(\pi, 1)$ in the \'etale topology and with Faltings' topos might want to skip ahead to \ref{intro:ab}.

\subsection{}
\label{intro:motivation}
When studying differentiable manifolds, one benefits from the fact that the underlying topological space is locally contractible. This is not the case in algebraic geometry: a smooth complex algebraic variety often does not admit a Zariski open cover by subvarieties which are contractible in the classical topology (a curve of nonzero genus, for example). 
On the other hand, as noticed by Artin in the course of the proof of the comparison theorem \cite[Exp. XI, 4.4]{Artin}, one can find a Zariski open cover by $K(\pi, 1)$ spaces (see \ref{artin} for the precise statement).  
Recall that a path connected topological space is called a $K(\pi, 1)$ space if its universal cover is weakly contractible. This is equivalent to the vanishing of all of its higher homotopy groups. 

\subsection{}
\label{intro:kpi1}
The notion of a $K(\pi, 1)$ space has a natural counterpart in algebraic geometry, defined in terms of \'etale local systems. Let $Y$ be a connected scheme with a geometric point $\bar y$. If $\Ff$ is a locally constant constructible abelian sheaf on $Y_\et$, the stalk $\Ff_{\bar y}$ is a representation of the fundamental group $\pi_1(Y, \bar y)$, and we have natural maps
\[
	\rho^i : H^i(\pi_1(Y, \bar y), \Ff_{\bar y}) \to H^i(Y_{\et}, \Ff).
\] 
We call $Y$ a \emph{$K(\pi, 1)$} if for every $n$ invertible on $Y$, and every $\Ff$ as above with $n\cdot \Ff = 0$, the maps $\rho^i$ are isomorphisms for all $i\geq 0$. See Section~\ref{s:kpi1} for a slightly more general definition and a discussion of this notion.

\subsection{}
\label{intro:faltings}
In a similar way as in Artin's comparison theorem, coverings by $K(\pi, 1)$ play a role in Faltings' approach to $p$-adic comparison theorems \cite{Faltings,FaltingsAlmost,Olsson}. Let $K$ be a finite extension of $\QQ_p$ and let $X_K$ be a smooth scheme over $K$. Loosely speaking, $p$-adic Hodge theory seeks to relate $H^i(X_{\bar K}, \QQ_p)$ to $H^i_{dR}(X_K/K)$ and other cohomology groups (see \cite{Fontaine}). In \cite{Faltings,FaltingsAlmost}, as a step towards this comparison, under the assumption that there is a smooth model $X/\Oo_K$, Faltings defines an intermediate cohomology theory $\Hh^\bullet(X)$ as the cohomology of a certain topos $\tilde E$ (following Abbes and Gros \cite{AbbesGros}, we call it the \emph{Faltings' topos}). This is the topos associated to a site $E$ whose objects are morphisms $V\to U$ over $X_{\bar K}\to X$ with $U\to X$ \'etale and $V\to U_{\bar K}$ finite \'etale (see \ref{def:faltingstopos} for the definition). The association $(V\to U)\mapsto V$ induces a morphism of topoi
\[ \Psi : X_{\bar K, \et} \to \tilde E.   \]  
To compare $H^i(X_{\bar K}, \QQ_p)$ and $\Hh^\bullet(X)$, the first step is to investigate the higher direct images $R^i \Psi_*$. In this direction, Faltings shows the following generalization of Artin's result (\cite[Lemma 2.1]{Faltings}, see \ref{faltings}): \emph{every point $x\in X$ has an open neighborhood $U$ for which $U_{\bar K}$ is a $K(\pi, 1)$}. It follows that \emph{$R^i \Psi_* \Ff = 0$ for $i>0$ and every locally constant constructible abelian sheaf $\Ff$ on $X_{\bar K}$}. It is these two results that we are going to generalize.

\subsection{}
\label{intro:ab}\label{setup}
Let $V$ be a discrete valuation ring with perfect residue field $k$ and fraction field $K$ of characteristic zero. Choose an algebraic closure $\bar K$ of $K$, and let  
\[ S = \Spec V, \quad s = \Spec k, \quad \eta = \Spec K, \quad \bar\eta = \Spec\bar K. \]

For a scheme $X$ over $S$ and an open subscheme $X^\circ \subseteq X$, we denote by $\tilde E$ the Faltings' topos of $X^\circ_{\bar\eta}\to X$ (see Definition \ref{def:faltingstopos}), and by $\Psi:X^\circ_{\bar\eta,\et}\to \tilde E$ the morphism of topoi \ref{def:faltingstopos}(c). Consider the following four statements:

{\bf (A)} \emph{$X$ has a basis of the \'etale topology consisting of $U$ for which $U\times_X X^\circ_{\bar\eta}$ is a $K(\pi, 1)$},

{\bf (B)} \emph{for every geometric point $\bar x\in X$, $X_{(\bar x)}\times_X X^\circ_{\bar\eta}$ is a $K(\pi, 1)$}, 

{\bf (C)} \emph{$R^i \Psi_* \Ff = 0$ ($i>0$) for every locally constant constructible abelian sheaf $\Ff$ on $X^\circ_{\bar\eta}$},

{\bf (D)} \emph{for every locally constant constructible abelian sheaf $\Ff$ on $X^\circ_{\bar\eta}$,  the natural maps}
\[ H^i(\tilde E, \Psi_* (\Ff)) \to H^i(X^\circ_{\bar\eta,\et}, \Ff). \]
\hspace{1cm} \emph{are isomorphisms for all $i\geq 0$.}

\noindent Then $(A)\Rightarrow (B)\Rightarrow (C)\Rightarrow (D)$, and the aforementioned theorem of Faltings (\ref{faltings}) states that (A) holds if $X$ is smooth over $S$ (and $X^\circ = X$). Faltings has also shown \cite[Lemma 2.3]{Faltings} that (B) is true if $X$ is smooth over $S$ and $X^\circ$ is the complement of a normal crossings divisor relative to $S$.

It is natural to ask whether these two results remain true if we do not require that $X$ be smooth over $S$ (we still want $X_{\eta}$, or at least $X^\circ_\eta$, to be smooth over $\eta$).  In general, the answer is no, even for $X$ regular (see Section~\ref{s:Lodh-counterex} for a counterexample). Note that the scheme $(X_{(\bar x)})_{\bar\eta}$ in (B) is the algebraic analogue of the Milnor fiber.

\subsection{}
\label{intro:main}
The most natural and useful generalization, hinted at in \cite[Remark on p. 242]{FaltingsAlmost}, and brought to our attention by Ahmed Abbes, seems to be the case of $X$ log smooth over $S$, where we endow $S$ with the ``standard'' log structure $\Mm_S\to \Oo_S$, i.e. the compactifying log structure induced by the open immersion $\eta \hookrightarrow S$. 
Our first main result confirms this expectation:

\begin{theorem*}[\ref{main}]
Assume that $\chara k=p>0$. Let $(X, \Mm_X)$ be a log smooth log scheme over $(S, \Mm_S)$ such that $X_\eta$ is smooth over $\eta$. Then (A) holds for $X^\circ=X$. 
\end{theorem*} 

Note that in the applications, in the above situation one usually cares about the case $X^\circ = (X, \Mm_X)_{tr}$ (the biggest open on which the log structure is trivial). While the theorem deals with $X^\circ=X$, we are able to deduce corollaries about the other case as well (see the next section).

The strategy is to reduce to the smooth case (idea due to R. Lodh) by finding an \'etale neighborhood $U'$ of $\bar x$ in $X$ and a map
\[ f:U' \to W'  \]
to a smooth $S$-scheme $W'$ such that $f_\eta:U'_\eta\to W'_\eta$ is finite \'etale. In such a situation, by Faltings's result (\ref{faltings}), there is an open neighborhood $W$ of $f(x)$ ($x$ being the underlying point of $\bar x$) in $W'$ such that $W_\eta$ is a $K(\pi, 1)$. Then $U = f^{-1}(W)$ is an \'etale neighborhood of $\bar x$, and since $U_\eta\to W_\eta$ is finite \'etale, $U_\eta$ is a $K(\pi, 1)$ as well.

The proof of the existence of $f$ makes use of the technique of Nagata's proof of the Noether Normalization Lemma, combined with the observation that the exponents used in that proof can be taken to be divisible by high powers of $p$ (see \ref{normalization}). Therefore our proof applies only in mixed characteristic. While we expect the result to be true regardless of the characteristic, we point out an additional difficulty in equal characteristic zero in \ref{counterex}.

\subsection{}
\label{intro:comparison}
We also treat the equicharacteristic zero case and the case with boundary. More precisely, we use Theorem~\ref{main} and log absolute cohomological purity to prove the following:

\begin{theorem*}[\ref{comparison}+\ref{comparisoncor}]
Let $(X, \Mm_X)$ be a log smooth log scheme over $(S, \Mm_S)$ such that $X_\eta$ is smooth over $\eta$, and let $X^\circ = (X, \Mm_X)_{tr}$ be the biggest open subset on which $\Mm_X$ is trivial. If $\chara k = 0$, assume moreover that $(X, \Mm_X)$ is saturated. Then (B)--(D) above hold for $X$ and $X^\circ$.
\end{theorem*}

\subsection{Outline}
Sections~\ref{s:functoriality}--\ref{s:loggeom} are preliminary. Section~\ref{s:functoriality} contains some abstract nonsense on cohomology groups of topoi, which is then used in Section~\ref{s:kpi1} where we review the notion of a $K(\pi, 1)$ scheme. Section~\ref{s:loggeom} provides a review of the relevant logarithmic geometry. 

Sections~\ref{s:noether}--\ref{s:Lodh-counterex} constitute the heart of the paper. Section~\ref{s:noether} deals with a variant of Noether normalization and proves the key Proposition \ref{preet-normal}. The proof of our main Theorem~\ref{main} is the subsequent Section~\ref{s:main}. Section~\ref{s:Lodh-counterex} gives an example of a regular scheme for which the assertion of Theorem~\ref{main} does not hold.

Section~\ref{s:nakayama} deals with the equicharacteristic zero case. The final Section~\ref{s:comparison} reviews the definition of Faltings' topos and proves our second main result, the comparison Theorem~\ref{comparison}.

\subsection{Acknowledgements}
I am indebted to A. Abbes for suggesting the problem, and to him, my advisor A. Ogus, and to M. Olsson for numerous discussions and suggestions, and for careful reading of several versions of the manuscript. Part of this work has been done during my visit to the IH\'ES in the summer of 2014. I would like to thank the institution for the hospitality, and again A. Abbes for the invitation. This work was supported by a Fulbright International Science and Technology Award.

\section{Functoriality properties of the cohomology pull-back maps}
\label{s:functoriality}

This section checks a certain functoriality property of cohomology of topoi, needed in Section~\ref{s:kpi1}. The reader should feel no discomfort in skipping this part. The result we need is that given a commutative diagram of topoi 
\[
		\xymatrix{
			X' \ar[r]^{g'} \ar[d]_{f'}  & X \ar[d]^f \\
			Y'\ar[r]_g                       & Y
		}
\]
and a sheaf $\Ff$ on $X$, there exist certain natural commutative diagrams \eqref{commdiag}
\[			
	\xymatrix{
			H^i(Y, f_* \Ff) \ar[d] \ar[r] & H^i(X, \Ff) \ar[d] \\
			H^i(Y', g^* f_*\Ff)\ar[r]  & H^i(X', g'^* \Ff)
	}
\]
for all $i\geq 0$. 

\subsection{}\label{funct:setup} Let 
\[
		\xymatrix{
			X' \ar[r]^{g'} \ar[d]_{f'}  & X \ar[d]^f \\
			Y'\ar[r]_g                       & Y
		}
\] 
be a commutative diagram of morphisms of topoi, that is, there is a chosen isomorphism 
\begin{equation} \label{iota1}
	\iota : f_* g'_* \isom g_* f'_*.   
\end{equation}
By adjunction, this also induces an isomorphism (also denoted $\iota$)
\begin{equation} \label{iota2}
 	\iota : f'^* g^* \isom g'^* f^* .   
\end{equation}

\subsection{}\label{bc}
Applying $f_*$ to the unit $\eta: id\to g'_* g'^*$ and composing with \eqref{iota1} yields a map
\[ f_* \to f_* g'_* g'^* \isom g_* f'_* g'_*, \]
which (using the adjunction between $g^*$ and $g_*$) gives us a map
\begin{equation}\label{bcdef}
	b.c. : g^* f_* \to f'_* g'^*  
\end{equation}
called the base change map.

Similarly, applying $g'^*$ to the counit $\varepsilon : f^* f_*\to id$, and composing with \eqref{iota2} yields a map
\[ f'^* g^* f_* \isom g'^* f^* f_* \to g'^*, \]
which (using the adjunction between $f'^*$ and $f'_*$) gives us another map
\[ g^* f_* \to f'_* g'^*  \]
that is equal to \eqref{bcdef} by \cite[Exp. XVII, Proposition 2.1.3]{Artin}.

\subsection{}\label{pullback}
Recall that if $f:X\to Y$ is a morphism of topoi, there is a natural map of $\delta$-functors from the category of abelian sheaves on $Y$ to the category of abelian groups:
\[  f^* : H^i(Y, -)\to H^i(X, f^*(-)).  \]   
Indeed, the right hand side is a $\delta$-functor because $f^*$ is exact, the transformation is defined for $i=0$, and $H^i(Y, -)$ is a universal $\delta$-functor.

The formation of this map is compatible with composition, that is, if $g:Z\to X$ is another map, the diagram (of $\delta$-functors of the above type)
\[ 
		\xymatrix{
			H^i(Y, -) \ar[r]^{f^*} \ar[d]_{(fg)^*}  & H^i(X, f^*(-)) \ar[d]^{g^*} \\
			H^i(Z, (fg)^*(-)) \ar@{=}[r]                       & H^i(Z, g^* f^* (-))
		}
\]
commutes. 

\subsection{} \label{comm1}
Applying \ref{pullback} to the situation of \ref{funct:setup} and composing with the map induced by the counit $\varepsilon: f^* f_* \to id$, we get a system of natural transformations
\begin{equation}\label{defmu} \mu : H^i(Y, f_*(-)) \xto{f^*} H^i(X, f^* f_* (-)) \xto{\varepsilon} H^i(X, -).  
\end{equation}
These coincide with the edge homomorphisms in the Leray spectral sequence for $f$ by \cite[{$0_{III}$ 12.1.7}]{EGA_III_1}. 

Let $\mu'$ be the following composition:
\[ \mu' : H^i(Y', g^* f_* (-)) \xto{b.c.} H^i(Y', f'_* g'^*(-))\xto{f'^*} H^i(X', f'^* f'_* g'^*(-)) \xto{\varepsilon'} H^i(X', g'^*(-)), \]
where $\varepsilon': f'^* f'_* \to id$ is the counit.

The goal of this section is to show that the diagram
	\begin{equation}\label{commdiag}
			\xymatrix{
			H^i(Y, f_* (-)) \ar[d]_{g^*} \ar[r]^{\mu} & H^i(X, (-)) \ar[d]^{g'^*} \\
			H^i(Y', g^* f_*(-))\ar[r]_{\mu'}  & H^i(X', g'^* (-))
			}
	\end{equation}	
commutes.

\subsection{} The assertion of \ref{comm1} will follow from the commutativity of the following diagram
\[ 
	\xymatrix{
		H^i(Y, f_* (-)) \ar[dd]_{g^* \; \hskip 1.7in \mathtt{(I)} \hskip -1.8in} \ar[rr]^{f^*} & & H^i(X, f^* f_* (-)) \ar[d]_{g'^*  \; \hskip .8in \mathtt{(III)} \hskip -.9in} \ar[r]^\varepsilon & H^i(X, -) \ar[d]^{g'^*} \\
		& & H^i(X', g'^* f^* f_* (-)) \ar@{=}[dl]_\iota \ar[r]_{g'^*(\varepsilon)} & H^i(X', g'^* (-)) \\
		H^i(Y', g^* f_* (-)) \ar[d]_{b.c. \; \hskip .8in \mathtt{(II)} \hskip -.9in} \ar[r]^{f'^*} & H^i(X', f'^* g^* f_*(-)) \ar[d]^{f'^*(b.c.)} & {}^\mathtt{(IV)} & \\
		H^i(Y', f'^* g'^* (-)) \ar[r]_{f'^*} & H^i(X', f'^* f'_* g'^*(-)) \ar@/_2pc/[uurr]_{\varepsilon'}. & &
	}
\]
Square (I) commutes by the functoriality of $f^*$ (\ref{pullback}). Squares (II) and (III) commute simply because $f^*$ is a natural transformation.

\subsection{} It remains to prove that (IV) commutes. This in turn will follow from the commutativity of
\[
	\xymatrix{
		f'^* g^* f_* \ar@{=}[r]^\iota \ar[d]_{f'^*(b.c.)} & g'^* f^* f_* \ar[d]^{g'^* (\varepsilon)} \\
		f'^* f'_* g'^* \ar[r]_{\varepsilon'} & g'^* .
	}
\]
By the discussion of \ref{bc}, the composition $g'^*(\varepsilon)\circ \iota$ above is adjoint (under the adjunction between $f'^*$ and $f'_*$) to the base change map $b.c.: g^* f_* \to f'_* g'^* $. It suffices to show that $\varepsilon\circ f'^*(b.c.)$ is also adjoint to the base change map. This follows precisely from the triangle identities for the adjunction between $f'^*$ and $f'_*$.

\section{\texorpdfstring{$K(\pi, 1)$ schemes}{K(pi, 1) schemes}}
\label{s:kpi1}

This section recalls the definition of a $K(\pi, 1)$ space in algebraic geometry, establishes some basic properties that apparently do not appear in the literature, and states the theorems of Artin and Faltings which assert the existence of coverings of smooth schemes by $K(\pi, 1)$'s.

\subsection{}\label{conncomp}
We will often consider schemes which are coherent and have a finite number of connected components (see \cite[9.6]{AbbesGros} for some criteria).

\subsection{}
Let $Y$ be a scheme satisfying \eqref{conncomp}. We denote by $\mathrm{F\et}(Y)$ the full subcategory of the \'etale site $\mathrm{\acute{E}t}(Y)$ consisting of \emph{finite} \'etale maps $Y'\to Y$, endowed with the induced topology, and by $Y_{f\et}$ the corresponding topos (cf. \cite[9.2]{Higgs2}). Note that the maps in $\mathrm{F\et}(Y)$ are also finite \'etale. The inclusion functor induces a morphism of topoi (cf. \cite[9.2.1]{Higgs2})
\[ \rho : Y_\et \to Y_{f\et}. \]
The pullback $\rho^*$ identifies $Y_{f\et}$ with the category of sheaves on $Y_\et$ equal to the union of their locally constant subsheaves (cf. \cite[5.1]{Olsson}, \cite[9.17]{Higgs2}). 
If $Y$ is connected and $\bar y\to Y$ is a geometric point, we have an equivalence of topoi
$Y_{f\et} \isom B\pi_1(Y, \bar y)$.

\begin{definition}[{cf. \cite[Definition 5.3]{Olsson}}]\label{kpi1def}
Let $\pp$ be a set of prime numbers. A scheme $Y$ satisfying \ref{conncomp} is called a \emph{$K(\pi, 1)$ for $\pp$-adic coefficients} if for every integer $n$ with $\ass n \subseteq \pp$, and every sheaf of $\ZZ/(n)$-modules $F$ on $Y_{f\et}$, the natural map
\[ F \to R\rho_* \rho^* F \]
is an isomorphism. If $\pp$ is the set of primes invertible on $Y$, we simply call $Y$ a $K(\pi, 1)$.
\end{definition}

The above condition is equivalent to saying that if $\Ff$ is a locally constant constructible sheaf of $\ZZ/(n)$-modules on $Y_\et$, then $R^i \rho_* \Ff = 0$ for $i>0$ (cf. \cite[9.17]{Higgs2}).

\begin{proposition}\label{kpi1prop} Let $\pp$ be a set of prime numbers, and let $Y$ be a scheme satisfying \ref{conncomp}. 
\begin{enumerate}[(a)]
	\item $Y$ is a $K(\pi, 1)$ for $\pp$-adic coefficents if and only if for every integer $n$ with $\ass n\subseteq \pp$, every locally constant constructible sheaf of $\ZZ/(n)$-modules $\Ff$ on $Y_\et$, and every class $\zeta\in H^i(Y, \Ff)$ with $i>0$, there exists a finite \'etale surjective map $f:Y'\to Y$ such that $f^*(\zeta) = 0 \in H^i(Y', f^* \Ff)$.
	\item Let $f:X\to Y$ be a finite \'etale surjective map. Then $X$ satisfies \ref{conncomp}, and $Y$ is a $K(\pi, 1)$ for $\pp$-adic coefficients if and only if $X$ is.
	\item Suppose that $Y$ is of finite type over a field $F$ and that $F'$ is a field extension of $F$. Denote $X = Y_{F'}$. Then $X$ satisfies \ref{conncomp}, and $Y$ is a $K(\pi, 1)$ if and only if $X$ is.
\end{enumerate}
\end{proposition}

\begin{proof}
(a) Let $\Ff$ be a locally constant constructible sheaf of $\ZZ/(n)$
-modules on $Y_\et$. Then $R^i \rho_* \Ff$ is the sheaf of $\mathrm{F\et}(Y)$ associated to the presheaf
\[ (f:X\to Y) \quad \mapsto \quad H^i(X, f^*\Ff).   \]
	It follows that $R^i \rho_* \Ff = 0$ if and only if the following condition (a') holds: \emph{for every finite \'etale $(f:X\to Y)\in \mathrm{F\et}(Y)$ and every $\zeta\in H^i(X, f^*\Ff)$, there exists a cover $\{g_i:(f_i:X_i\to Y)\to (f:X\to Y)\}_{i\in I}$ such that $f_i^* \zeta = 0 \in H^i(X_i, g_i^* f^* \Ff) = H^i(X_i, f_i^* \Ff)$}. In case $Y$ is connected, each $g_i$ with $X_i$ nonempty is finite \'etale surjective, hence in such case (a') implies (a) by considering $Y'=Y$. The general case follows by considering the connected components of $Y$ separately. 

	We prove that the condition in (a) implies (a'). In the situation of (a'), let $\Ff_0 = f^* \Ff$ for brevity, and consider the sheaf $f_* \Ff_0$. As $f$ is finite \'etale, $f_* \Ff_0$ is locally constant constructible and $R^j f_* \Ff_0 = 0$ for $j>0$, therefore the natural map \eqref{defmu}
	\begin{equation}\label{natiso} 
		\mu:H^i(Y, f_* \Ff_0) \to H^i(X, \Ff_0),  
	\end{equation}
	is an isomorphism. Let $\zeta'\in H^i(Y, f_* \Ff)$ map to $\zeta$ under \eqref{natiso}. By (a), there exists a finite \'etale surjective map $g:Y'\to Y$ with $g^* \zeta' = 0 \in H^i(Y', g^* f_* \Ff_0)$. Form a cartesian diagram
	\[
		\xymatrix{
			X' \ar[r]^{g'} \ar[d]_{f'}  & X \ar[d]^f \\
			Y'\ar[r]_g                       & Y.
		}
	\] 
	Then $g'$ is finite \'etale and surjective. Moreover, by \ref{comm1}, the diagram 
	\[
		\xymatrix{
		    & \zeta' \ar@{|->}[r] \ar@{}[d]|{\indown} & \zeta \ar@{}[d]|{\indown} \\
		\zeta' \ar@{|->}[d] \ar@{}[r]|{\displaystyle\in \hskip .4in} & 	H^i(Y, f_* \Ff_0) \ar[d]_{g^*} \ar[r]^{\mu} & H^i(X, \Ff_0) \ar[d]^{g'^{*}} \\
		0 \ar@{}[r]|{\displaystyle\in \hskip .4in} &	H^i(Y', g^* f_*\Ff_0)\ar[r]_{\mu'}  & H^i(X', g'^* \Ff_0)
		}
	\]
	commutes, hence $g'^* \zeta = 0$.

	(b) The argument is similar to (a). If $X$ is a $K(\pi, 1)$, $Y$ is a $K(\pi, 1)$ as well by the characterization of (a). Suppose that $Y$ is a $K(\pi, 1)$ and let $\Ff_0$ be a locally constant constructible sheaf of $\ZZ/(n)$-modules on $X_\et$, $\zeta\in H^i(X, \Ff_0)$ ($i>0$). Apply the same reasoning as in (a).

	(c) If $F'/F$ is a finite separable extension, this follows from (b) as then $X\to Y$ is finite \'etale and surjective. If $F'$ is a separable closure of $F$, the assertion follows from the characterization in (a) and usual limit arguments. If $F'/F$ is finite and purely inseparable, $X\to Y$ induces equivalences $X_\et\isom Y_\et$ and $X_{f\et}\isom Y_{f\et}$, so there is nothing to prove.  If $F'$ and $F$ are both algebraically closed, the assertion follows from \cite[Exp. XVI, 1.6]{Artin}. If $F'/F$ is arbitrary, pick an algebraic closure $\bar F'$ and let $\bar F$ be the algebraic closure of $F$ in $\bar F'$. We now have a ``path'' from $F$ to $F'$ of the form
	\[ F\subseteq F^{sep} \subseteq \bar F \subseteq \bar F' \supseteq (F')^{sep} \supseteq F' \]
	and the assertion follows from the preceding discussion.
\end{proof}

\begin{theorem}[{M. Artin, follows from \cite[Exp. XI, 3.3]{Artin}, cf. \cite[Lemma 5.5]{Olsson}}]
\label{artin}
Let $Y$ be a smooth scheme over a field of characteristic zero, $y$ a point of $Y$. There exists an open neighborhood $U$ of $y$ which is a $K(\pi, 1)$.
\end{theorem}


\begin{theorem}[{Faltings, \cite[Lemma 2.1]{Faltings}, cf. \cite[Theorem 5.4]{Olsson}}]
\label{faltings}
Let $S$ be as in \ref{setup}, let $Y$ be a smooth $S$-scheme, and let $y$ be a point of $Y$. There exists an open neighborhood $U$ of $y$ for which $U_\eta$ is a $K(\pi, 1)$.
\end{theorem}

\section{Some logarithmic geometry}\label{s:loggeom}

In this section, we review the relevant facts from log geometry and investigate the local structure of a log smooth $S$-scheme (with the standard log structures on $X$ and $S$). We also state the logarithmic version of absolute cohomological purity, used in Section \ref{s:nakayama}.

\subsection{Conventions about log geometry}

If $P$ is a monoid, $\bar P$ denotes the quotient of $P$ by its group $P^*$ of invertible elements, and $P\to P^{gp}$ is the universal (initial) morphism form $P$ into a group. $P$ is called \emph{fine} if it is finitely generated and integral (i.e., $P\to P^{gp}$ is injective). A \emph{face} of a monoid $P$ is a submonoid $F\subseteq P$ satisfying $x+y\in F \Rightarrow x, y\in F$. For an integral monoid $P$ and face $F$, the \emph{localization} of $P$ at $F$ is the submonoid $P_F$ of $P^{gp}$ generated by $P$ and $-F$. It satisfies the obvious universal property. If $Q$ is a submonoid of an integral monoid $P$, the quotient $P/Q$ is defined to be the image of $P$ in $P^{gp}/Q^{gp}$. 

For a monoid $P$, $\AA_P = \Spec (P\to \ZZ[P])$ is the log scheme associated to $P$; for a homomorphism $\theta:P\to Q$, $\AA_\theta:\AA_Q\to\AA_P$ is the induced morphism of log schemes.  A morphism $(X, \Mm_X)\to (Y, \Mm_Y)$ of log schemes is \emph{strict} if the induced map $f^\flat: f^* \Mm_Y \to \Mm_X$ is an isomorphism. A strict map to some $\AA_P$ is called a \emph{chart}. A log scheme is \emph{fine} if \'etale locally it admits a chart with target $\AA_P$ for a fine monoid $P$. If $j:U\to X$ is an open immersion, the \emph{compactifying log structure} on $X$ associated to $U$ is the preimage of $j_* \Oo_U^*$ under the restriction map $\Oo_X\to j_* \Oo_U$. For a log scheme $(X, \Mm_X)$, we denote by $(X, \Mm_X)_{tr}$ the complement of the support of $\bar \Mm_X$. It is an open subset of $X$ if $(X, \Mm_X)$ is fine, the biggest open subset on which $\Mm_X$ is trivial. \emph{References:} \cite{Kato,Ogus,HofM}. 



We recall Kato's structure theorem for log smooth morphisms (which for our purposes might as well serve as a definition):

\begin{theorem}[{\cite[Theorem 3.5]{Kato}}] \label{kato}
Let $f:(X, \Mm_X)\to (S, \Mm_S)$ be a morphism of fine log schemes. Assume that we are given a chart $\pi:(S, \Mm_S)\to \AA_Q$ with $Q$ a fine monoid. Then $f$ is log smooth if and only if, \'etale locally on $X$, there exists a fine monoid $P$, a map $\rho: Q\to P$ such that the kernel and the torsion part of the cokernel of $\rho^{gp}:Q^{gp}\to P^{gp}$ are finite groups of order invertible on $S$, and a commutative diagram 
\[ 
     \xymatrix{
     	(X, \Mm_X) \ar[r] \ar[dr]_f & \AA_{P,\rho,\pi} \ar[r]\ar[d]  &  \AA_P  \ar[d]^{\AA_\rho} \\
     	           & (S,\Mm_S)    \ar[r]_\pi              &  \AA_Q,   
     } 
\]
where the square is cartesian and $(X, \Mm_X)\to \AA_{P, \rho, \pi}$ is strict and \'etale (as a morphism of schemes). 
\end{theorem}

\subsection{} \label{chart}
Suppose that $f:(X, \Mm_X)\to \AA_P$ is a chart with $P$ a fine monoid, and $\bar x\to X$ is a geometric point. Let $F\subseteq P$ be the preimage of $0$ under the induced homomorphism $P\to \bar \Mm_{X,\bar x}$. Then $F$ is a face of $P$, and $P$ injects into the localization $P_F$. Moreover, the induced map $P/F\to \bar \Mm_{X,\bar x}$ is an isomorphism.

As $P$ is finitely generated, $F$ is finitely generated as a face, hence the natural map $\AA_{P_F}\to \AA_P$ is an open immersion: if $F$ is generated as a face by an element $a\in P$, then $\AA_{P_F} = D(a)$. Let us form a cartesian diagram
\[
	\xymatrix{
		(U, \Mm_X|_U) \ar[d] \ar[r]  & \AA_{P_F} \ar[d] \\
		(X, \Mm_X) \ar[r]_f            & \AA_P.
 	}
\]
Then $U = D(f^* a)$ and $\bar x$ lies in $U$ because $f^* a$ is an invertible element of $\Oo_{X, x}$ by the construction of $F$. 

It follows that any chart $f:(X, \Mm_X)\to \AA_P$ as above can be locally replaced by one for which the homomorphism $\bar P\to \bar \Mm_{X,\bar x}$ is an isomorphism, without changing the local properties of $f$ (e.g. without sacrificing \'etaleness if $f$ is \'etale). 

\subsection{} \label{our-setting}
In the situation of \ref{setup}, let $f:(X, \Mm_X)\to (S, \Mm_S)$ be a log smooth morphism. 
 Applying Theorem \ref{kato} to a chart $\pi: (S, \Mm_S)\to \AA_\NN$ given by a uniformizer $\pi$ of $V$, we conclude that, \'etale locally on $X$, there exists a strict \'etale $g:(X, \Mm_X)\to \AA_{P, \rho, \pi}$ where
\begin{equation}\label{defaprp} 
	\AA_{P, \rho, \pi} = \Spec\left( P\to\frac{V[P]}{(\pi - \rho)}\right) 
\end{equation}
for a fine monoid $P$ and a non-invertible element $\rho \in P$ with the property that $(P/\rho)^{gp}$ is $p$-torsion free. 

\subsection{} \label{our2}
In the above situation, assume moreover that $X_\eta$ is smooth over $\eta$. Localizing $P$, we can assume that the scheme underlying $(\AA_{P,\rho,\pi})_\eta$ is smooth over $\eta$ as well. But $(\AA_{P,\rho,\pi})_\eta$ is isomorphic to $\Spec(P/\rho \to K[P/\rho])$. Note that for a fine monoid $M$, $\Spec K[M]$ is smooth over $K$ if and only if $\bar M$ is a free monoid. It follows that $\bar P/\rho$ is free, and therefore the stalks of $\Mm_{X/S} := \Mm_{X}/f^\flat \Mm_S$ are free monoids. Moreover, every geometric point $\bar x$ of $X$ has an \'etale neighborhood $U$ such that $(U_\eta, \Mm_{X}|_{U_\eta})_{tr}$ is the complement of a divisor with strict normal crossings on $U_\eta$.

\subsection{Absolute cohomological purity} \label{pur}
We will need the following result (cf. \cite[Proposition 2.0.2]{Nakayama}): let $(X, \Mm_X)$ be a regular \cite[Definition 2.1]{KatoToric} (in particular, fine and saturated) log scheme such that $X$ is locally noetherian. Let $X^\circ = (X, \Mm_X)_{tr}$ be the biggest open subset on which $\Mm_X$ is trivial, and let $j:X^\circ\to X$ be the inclusion. Let $n$ be an integer invertible on $X$. Then for any $q\geq 0$, we have a natural isomorphism
\begin{equation}\label{puriso}
	R^q j_* (\ZZ/n\ZZ) \isom \bigwedge\nolimits^q \bar \Mm^{gp}_X \otimes \ZZ/n\ZZ(-1). 
\end{equation}

\section{\texorpdfstring{$\eta$-\'etale maps and Noether normalization}{eta-\'etale maps and Noether normalization}}
\label{s:noether}

This section contains the key technical point used in the proof of Theorem \ref{main}. First, we prove a (slightly spiced-up) relative version of the Noether Normalization Lemma (Proposition \ref{nagatatrick}). Then we study $\eta$-\'etale maps $f:X\to Y$ over $S$, that is, maps which are \'etale in an open neighborhood of the closed fiber $X_s$ of $X$. The main result is Proposition \ref{preet-normal}, which asserts that \emph{in mixed characteristic} we can often replace an $\eta$-\'etale map $f':X\to \AA^d_S$ by a quasi-finite $\eta$-\'etale map $f:X\to \AA^d_S$.

\subsection{Relative Noether normalization}\label{normalization}

\begin{lemma}\label{degreetrick}
Let $F$ be a field, and let $a\in F[x_1, \ldots, x_n]$ be a nonzero polynomial. For large enough $m$, the polynomial 
\[ a(x_1 - x_n^{m}, x_2 - x_n^{m^2}, \ldots, x_{n-1} - x_n^{m^{n-1}}, x_n),\]
treated as a polynomial in $x_n$ over $F[x_1, \ldots, x_{n-1}]$, has a constant leading coefficient.
\end{lemma}

\begin{proof}
Standard, cf. e.g. \cite[\S 1]{Mumford} or \stacks{051N}.
\end{proof}

\begin{definition}
Let $f:X\to Y$ be a map of schemes over some base scheme $S$. We call $f$ \emph{fiberwise finite relatively to $S$} if for every point $s\in S$, the induced map $X_s \to Y_s$ is finite.
\end{definition}

Let $V, S, \ldots$ be as in \ref{setup} (the assumptions on $K$ and $k$ are unnecessary here). The following is a relative variant of Noether normalization. In the applications we will take $N$ to be a high power of $p$.

\begin{proposition} \label{nagatatrick}
Let $X=\Spec R$ be a flat affine $S$-scheme of finite type, let $d\geq 0$ be an integer such that $\dim (X/S)\leq d$, and let $x_1, \ldots, x_d \in R$. For any integer $N\geq 1$, there exist $y_1, \ldots, y_d \in R$ such that 
\begin{enumerate}[(i)]
	\item the map $f = (f_1, \ldots, f_d) : X\to \AA^d_S$, $f_i = x_i + y_i$, is fiberwise finite relatively to $S$.
	\item the $y_i$ belong to the subring generated by $N$-th powers of elements of $R$,
\end{enumerate}
\end{proposition}

\begin{proof}
Write $R = V[x_1, \ldots, x_d, x_{d+1}, \ldots, x_n]/I$. The proof is by induction on $n-d$. If $n=d$, then the map $(x_1, \ldots, x_d):X\to \AA^d_S$ is a closed immersion, and we can take $y_i = 0$. Suppose that $n>d$.

Let $a\in V[x_1, \ldots, x_n]$ be an element of $I$ with nonzero image in $k[x_1, \ldots, x_n]$. Such an element exists, for otherwise $X_s$ is equal to $\AA^n_s$, hence cannot be of dimension $\leq d$ as $n>d$. 

For an integer $m\geq 1$, consider the elements 
\[ z_i = x_i + x_n^{(Nm)^i}, \quad i=1, \ldots, n-1, \]
and let $R'\subset R$ be the $V$-subalgebra generated by $z_1, \ldots, z_{n-1}$. By Lemma \ref{degreetrick} applied to the image of $a$ in $K[x_1, \ldots, x_n]$ (with $F=K$) resp. $k[x_1, \ldots, x_n]$ (with $F=k$), there exists an $m$ such that the images of $x_n$ in $R\otimes K$, resp. $R\otimes k$ will be integral over $R'\otimes K$, resp. $R'\otimes k$. As $x_i = z_i - x_n^{(Nm)^i}$, the other $x_i$ will have the same property, which is to say, $\Spec R \to \Spec R'$ is fiberwise finite over $S$.

We check that it is possible to apply the induction assumption to $X'=\Spec R'$ and $z_1, \ldots, z_d \in R'$. Since $R'$ is a subring of $R$ and $R$ is torsion-free, $R'$ is torsion-free as well, hence flat over $V$. As $R'\otimes_V K\to R\otimes_V K$ is finite and injective, we have $\dim X'_\eta = \dim X_\eta$. Since $R'$ is flat over $V$, we have $\dim X'_s \leq \dim X'_\eta$, so $\dim X'_s\leq d$ as well. Finally, $R'$ is generated as a $V$-algebra by $n-1$ elements with $z_1, \ldots, z_d$ among them. 

By the induction assumption applied to $X'=\Spec R'$ and $z_1, \ldots, z_d \in R$, there exists a fiberwise finite map $f'=(f_1, \ldots, f_d):\Spec R' \to \AA^d_S$ with $f_i = z_i + y'_i$, where the $y'_i$ belong to the subring of $R'$ generated by $N$-th powers of elements of $R'$. As the composition of fiberwise finite maps is clearly fiberwise finite, the composition $f=(f_1, \ldots, f_n):X=\Spec R\to \AA^d_S$ is fiberwise finite (i). We have $f_i = x_i + y_i$, $y_i = y'_i + (x_n^{N^{i-1}m^i})^N$, so (ii) is satisfied as well.
\end{proof}

It would be interesting to have a generalization of this result to a general noetherian local base ring $V$.

\subsection{\texorpdfstring{$\eta$-\'etale maps}{eta-\'etale maps}}
We now assume that $\chara k = p > 0$. Let $f:X\to Y$ be a map of $S$-schemes of finite type.

\begin{definition}
We call $f$ \emph{$\eta$-\'etale at a point $x\in X_s$} if there is an open neighborhood $U$ of $x$ in $X$ such that $f_\eta : U_\eta\to Y_\eta$ is \'etale. We call $f$ \emph{$\eta$-\'etale} if it is $\eta$-\'etale at all points $x\in X_s$, or equivalently, if there is an open neighborhood $U$ of $X_s$ in $X$ such that $f_\eta:U_\eta\to Y_\eta$ is \'etale.
\end{definition}

We warn the reader not to confuse ``$f$ is $\eta$-\'etale'' with ``$f_\eta$ is \'etale'' (the latter is a stronger condition).

\begin{lemma}\label{condpreet}
Consider the following properties:
\begin{enumerate}[(i)]
	\item $f$ is $\eta$-\'etale,
	\item there exists an $n\geq 0$ such that $(p^n \Omega^1_{X/Y})|_{X_s} = 0$ (pull-back as an abelian sheaf),
	\item there exists an $n\geq 0$ such that $(p^n \Omega^1_{X/Y})\otimes_V k =0$.
\end{enumerate}
Then $(i)\Rightarrow (ii) \Leftrightarrow (iii)$, and the three properties are equivalent if $X_\eta$ and $Y_\eta$ are smooth of the same relative dimension $d$ over $S$. 
\end{lemma}

\begin{proof}
The equivalence of $(ii)$ and $(iii)$ follows from Nakayama's lemma.

Suppose that $f$ is $\eta$-\'etale, and let $U\subseteq X$ be an open subset containing $X_s$ such that $f|_{U_\eta}$ is \'etale. In particular, $f|_{U_\eta}$ is unramified, hence $\Omega^1_{X/Y}|_{U_\eta} = 0$.

Recall that if $\Ff$ is a coherent sheaf on a noetherian scheme $U$ and $a \in \Gamma(U, \Oo_U)$, then $\Ff|_{D(a)} = 0$ if and only if $a^n \Ff = 0$ for some $n\geq 0$. Applying this to $\Ff = \Omega^1_{X/Y}|_U$ and $a=p$ (noting that $U_\eta = D(p)$), we get that $(p^n \Omega^1_{X/Y})|_U = 0$, hence in particular $(p^n \Omega^1_{X/Y})|_{X_s} = 0$. 

Suppose now that $(p^n \Omega^1_{X/Y})|_{X_s} = 0$ for some $n\geq 0$. By Nakayama's lemma, $(p^n \Omega^1_{X/Y})_x = 0$ for every $x\in X_s$, hence there is an open subset $U$ containing $X_s$ such that $(p^n \Omega^1_{X/Y})|_U = 0$. As $p$ is invertible on $U_\eta$, we get that $\Omega^1_{X/Y}|_{U_\eta} = 0$, that is, $f|_{U_\eta}$ is unramified. If $X_\eta$ and $Y_\eta$ are smooth of the same dimension over $S$, this is enough to guarantee that $f|_{U_\eta}$ is \'etale. 
\end{proof}

\begin{lemma}\label{goodW}
Suppose that $f$ is closed and $\eta$-\'etale. There exists an open neighborhood $W$ of $Y_s$ in $Y$ such that $f: f^{-1}(W)_\eta \to W_\eta$ is \'etale.
\end{lemma}

\begin{proof}
Let $Z \subseteq X_\eta$ be the locus where $f_\eta$ is not \'etale. Then $Z$ is closed in $X$. Since $f$ is a closed, $f(Z)$ is closed in $Y$, and of course $f(Z)\cap Y_s=\emptyset$. Take $W = Y \setminus f(Z)$. 
\end{proof}

We now consider the case $Y = \AA^d_S$.

\begin{lemma} \label{preetmod}
Suppose that $f':X\to \AA^d_S$ is such that $(p^n \Omega^1_{f'})|_{X_s} = 0$ and that $y_1, \ldots, y_d\in \Gamma(X, \Oo_X)$ are polynomials in $p^{n+1}$-powers of elements of $\Gamma(X, \Oo_X)$. If $f = f'+(y_1, \ldots, y_d)$, then $(p^n \Omega^1_{f})|_{X_s} = 0$ as well.
\end{lemma}

\begin{proof}
Let $S_n = \Spec V/p^{n+1} V$, $X_n = X\times_S S_n$. The presentations
\[ \Oo^d_X \xto{df'_i} \Omega^1_{X/S} \to \Omega^1_{f'}\to 0,\quad \Oo^d_X \xto{df_i} \Omega^1_{X/S} \to \Omega^1_{f}\to 0  \]
give after base change to $S_n$ the short exact sequences
\[ \Oo^d_{X_n} \xto{df'_i} \Omega^1_{X_n/S_n} \to \Omega^1_{f'}/p^{n+1} \to 0,  \quad \Oo^d_{X_n} \xto{df_i} \Omega^1_{X_n/S_n} \to \Omega^1_{f}/p^{n+1}\to 0.  \]
By the assumption on the $y_i$, we have $dy_i \in p^{n+1}\Omega^1_{X/S}$, therefore the two maps $\Oo^d_{X_n} \to \Omega^1_{X_n/S_n}$ above are the same. It follows that $\Omega^1_{f}/p^{n+1} \isom \Omega^1_{f'}/p^{n+1}$. The assumption that $(p^n \Omega^1_{f'})|_{X_s} = 0$ means that $p^n\Omega^1_{f'}/\pi p^n\Omega^1_{f'}=0$ for a uniformizer $\pi$ of $V$. As $p = u\pi^e$ for a unit $u\in V$ and $e\geq 1$ an integer, we have $p^n \Omega^1_{f'}/p^{n+1} \Omega^1_{f'} = 0$. Since $\Omega^1_{f}/p^{n+1} \isom \Omega^1_{f'}/p^{n+1}$, the same holds for $\Omega^1_f$. We thus have $(p^{n}\Omega^1_f)|_{X_n} = 0$, hence $(p^n \Omega^1_{f})|_{X_s} = 0$.
\end{proof}

\begin{proposition}\label{preet-normal}
Assume that $X$ affine and flat over $S$, that $X_\eta$ is smooth of relative dimension $d$ over $S$, and that $f':X\to \AA^d_S$ is $\eta$-\'etale. There exists an $f:X\to \AA^d_S$ which is $\eta$-\'etale and fiberwise finite over $S$.
\end{proposition}

\begin{proof}
By Lemma \ref{condpreet}, there exists an $n$ such that $(p^n \Omega^1_{f'})|_{X_s} = 0$. Apply Proposition \ref{nagatatrick} to $x_i = f'_i$ and $N = p^{n+1}$, obtaining a fiberwise finite $f:X\to \AA^d_S$ which differs from $f'$ by some polynomials in $p^{n+1}$-powers. Then $f$ is $\eta$-\'etale by Lemma \ref{preetmod} and Lemma \ref{condpreet}.
\end{proof}

\section{\texorpdfstring{Existence of $K(\pi, 1)$-neighborhoods}{Existence of K(pi, 1)-neighborhoods}}
\label{s:main}

\begin{theorem} \label{main}
Assume that $\chara k = p > 0$. Let $(X, \Mm_X)$ be a log smooth log scheme over $(S, \Mm_S)$ such that $X_\eta$ is smooth over $\eta$, and let $\bar x\to X$ be a geometric point. There exists an \'etale neighborhood $U$ of $\bar x$ such that $U_\eta$ is a $K(\pi, 1)$.
\end{theorem} 

\noindent \emph{Proof.}

\subsubsection{}
If $\bar x$ is contained in $X_\eta$, the existence of such a neighborhood follows from Theorem \ref{artin}. We are therefore going to restrict ourselves to the case where $\bar x$ is a geometric point of the closed fiber $X_s$. The question being \'etale local around $X$, we are allowed to shrink $X$ around $\bar x$ if needed.

\subsubsection{}
Let $\pi$ be a uniformizer of $V$, inducing a chart $(S, \Mm_S)\to \AA_\NN$. By the discussion of \ref{our-setting}---\ref{our2}, in an \'etale neighborhood of $\bar x$ there exists a fine monoid $P$, an element $\rho \in P$ such that $\bar{P/\rho}$ is a free monoid, and an \'etale map 
\[ g: X\to \AA_{P, \rho, \pi} = \Spec \left(P\to \frac{V[P]}{(\pi - \rho)}\right) \]
over $S$. 

We can replace $X$ by an \'etale neighborhood of $\bar x$ for which the above data exist. Shrinking $X$ further, we can also assume that $X$ is affine. 

\subsubsection{}
Let us denote by $P[\rho^{-1}]$ the submonoid of $P^{gp}$ generated by $P$ and $\rho^{-1}$, and by $P/\rho$ the quotient of $P[\rho^{-1}]$ by the subgroup generated by $\rho$. Since $\overline{P[\rho^{-1}]}=\bar{P/\rho}$ is free, there is an isomorphism $P[\rho^{-1}] \isom P[\rho^{-1}]^* \oplus \overline{P[\rho^{-1}]}$. Picking an isomorphism $\overline{P[\rho^{-1}]}\isom \NN^b$ and a decomposition of $P[\rho^{-1}]^*$, we can write $P[\rho^{-1}]\isom T \oplus\ZZ\oplus \ZZ^a\oplus \NN^b$ where $T$ is a finite abelian group and $\rho$ corresponds to an element of the $\ZZ$ summand. Dividing by $\rho$, we obtain an isomorphism $P/\rho \isom T\oplus\ZZ^a\oplus \NN^b$. Let $d=a+b$, and let $\chi_0 : \NN^{d}=\NN^a\oplus\NN^b \to T\oplus \ZZ^a\oplus \NN^b \isom P/\rho$ be the map implied by the notation. As the source of $\chi_0$ is free and $P\to P/\rho$ is surjective, we can choose a lift $\chi:\NN^d\to P$ of $\chi_0$:
\[
     \xymatrix{
     	        & &  \NN^d \ar[d]^{\chi_0} \ar@{.>}[lld]_\chi \\
     	P\ar@{->>}[rr] & &  P/\rho.
     } 
\]
Then $\chi$ induces a map $h:\AA_{P, \rho, \pi} \to \AA^d_S$ over $S$. 

\subsubsection{} I claim that $h_\eta$ is \'etale. Note first that $h_\eta$ is the pullback under $\pi: \eta \to \AA_\ZZ$ of the horizontal map in the diagram
\[
	\xymatrix{
		\AA_{P[\rho^{-1}]}  \ar[dr] \ar[rr]^{\AA_{\rho\oplus \chi}} & & \AA_\ZZ \times \AA_{\NN^d} \ar[dl] \\ 
		& \AA_\ZZ & 
	}
\]
We want to check that the horizontal map becomes \'etale after base change to $\QQ$. Since the base is $\AA_\ZZ = \mathbb{G}_m$ and the map is $\mathbb{G}_m$-equivariant, it suffices to check this on one fiber. If we set $\rho = 1$, the resulting map is none other than the map induced by 
\[  \NN^d =\NN^a\oplus\NN^b \hookrightarrow \ZZ^a\oplus\NN^b \hookrightarrow T\oplus \ZZ^a \oplus \NN^b \isom P/\rho = P[\rho^{-1}]/\rho,  \]
which is \'etale after adjoining $1/\# T$.

\subsubsection{}
Let $f'= h\circ g: X\to \AA^d_S$. This map is $\eta$-\'etale, therefore by Proposition \ref{preet-normal}, there exists a map $f : X\to \AA^d_S$ which is $\eta$-\'etale and fiberwise finite over $S$ (hence quasi-finite). 
%
%
Let $\bar y\to \AA^d_S$ be the image of $\bar x$ under $f$. 
As $f$ is quasi-finite, we can perform an \'etale localization at $\bar x$ and $\bar y$ which will make it finite. More precisely, we apply \cite[Th\'eor\`eme 18.12.1]{EGA_IV_4} (or \stacks{02LK}) and conclude that there exists a commutative diagram
\[
	\xymatrix{ 
		\bar x \ar[d] \ar[r]  & U' \ar[d]_{f} \ar[r] & X \ar[d]_f \\
		\bar y \ar[r]         & W' \ar[r]        & \AA^d_S
	}
\]
with $U'\to X'$ and $W'\to \AA^d_S$ \'etale and $f:U'\to W'$ finite. It follows that $f:U'\to W'$ is also $\eta$-\'etale.

\subsubsection{}
By Lemma \ref{goodW} applied to $f:U'\to W'$, we can shrink $W'$ around $\bar y$ (and shrink $U'$ accordingly to be the preimage of the new $W'$) so that $U'_\eta\to W'_\eta$ is finite \'etale.

\subsubsection{}
Since $W'$ is smooth over $S$, by Faltings' theorem (\ref{faltings}) there is an open neighborhood $W$ of $\bar y$ in $W'$ such that $W_\eta$ is a $K(\pi, 1)$. Let $U$ be the preimage of $W$ in $U'$ under $f:U'\to W'$. The induced map $f_\eta : U_\eta \to W_\eta$ is finite \'etale, hence $U_\eta$ is a $K(\pi, 1)$ as well by Proposition \ref{kpi1prop}(b).  \hfill $\qed$

\subsection{Remark}\label{relsmooth}
A reader familiar with the notion of a relatively smooth log structure (cf. \cite[Definition 3.6]{NakayamaOgus}, \cite{OgusDwork}) might appreciate the fact that the above proof applies to relatively smooth $X/S$ as well. Recall that we call $(X,\Ff)/(S, \Mm_S)$ \emph{relatively log smooth} if, \'etale locally on $X$, there exists a log smooth log structure $(X, \Mm)/(S, \Mm_S)$ and an inclusion $\Ff\subseteq \Mm$ as a finitely generated sheaf of faces, for which the stalks of $\Mm/\Ff$ are free monoids. We can then apply Theorem \ref{main} to $(X, \Mm)$ instead of $(X, \Ff)$. 

Important examples of relatively log smooth $X/S$ appear in the Gross--Siebert program in mirror symmetry \cite{GrossSiebert1,GrossSiebert2,GrossSiebert3} as so-called \emph{toric degenerations}. Degenerations of Calabi-Yau hypersurfaces in toric varieties are instances of such. For example, the \emph{Dwork families}
\[ X = \Proj V[x_0, \ldots, x_n]\Big/( (n+1)x_0\cdot\ldots\cdot x_n - \pi\cdot \sum_{i=0}^{n} x_i^{n+1} )  \]
(with the standard compactifying log structure) are relatively log smooth over $S$ if $n+1$ is invertible on $S$, but not log smooth for $n>2$ (\cite[Proposition 2.2]{OgusDwork}). 

\subsection{Obstacles in characteristic zero}
\label{counterex}

The need for the positive residue characteristic assumption in our proof of Theorem \ref{main} can be traced down to the application of Proposition \ref{preet-normal}: one can perform relative Noether normalization on an $\eta$-\'etale map $f':X\to \AA^d_S$ without sacrificing $\eta$-\'etaleness.  One might think that this is too crude and that one could replace that part with a Bertini-type argument. After all, we only need $\eta$-\'etaleness at one point! Unfortunately, this is bound to fail in characteristic zero even in the simplest example, that of a semistable curve:

\begin{proposition}
Let $X$ be an open subset of $\Spec V[x, y]/(xy - \pi) \subseteq \AA^2_S$ containing the point $P=(0,0)\in\AA^2_k$ and let $f:X\to \AA^1_S$ be an $S$-morphism. If $f_\eta : X_\eta\to \AA^1_\eta$ is \'etale, then $df$ is identically zero on one of the components of $X_s$. In particular, if $\chara k = 0$, then $f$ has to contract one of the components of $X_s$, hence is not quasi-finite.
\end{proposition}

\begin{proof}
Let $Z$ be the support of $\Omega^1_{X/\AA^1_S}$. As $f_\eta$ is \'etale, $Z \subseteq X_s$. On the other hand, the short exact sequence
\[ \Oo_X^2 \xto{ \left[ \begin{array}{cc} f_x & y \\ f_y & x \end{array} \right] } \Oo_X^2 \xto{\left[\begin{array}{cc} dx & dy \end{array} \right] } \Omega^1_{X/\AA^1_S} \to 0   \]
shows that $Z$ is the set-theoretic intersection of two divisors in $\AA^2_S$ (given by the equations $xy=\pi$ and $xf_x-yf_y = 0$), each of them passing through $P$, for if $g\cdot \Omega^1_{X/\AA^1_S} = 0$, then 
\[ \left[\begin{array}{cc} f_x & y \\ f_y & x \end{array} \right] \cdot C = \left[ \begin{array}{cc} g & 0 \\ 0 & g \end{array} \right]   \]
for some matrix $C$, hence $g^2 = (xf_x - yf_y)\cdot \det(C) \in (xf_x-yf_y)$.
Since $\AA^2_S$ is regular, by the dimension theorem we know that each irreducible component of $Z$ passing through $P$ has to be of positive dimension. Therefore $Z$ has to contain one of the components of $X_s$.
\end{proof}

\section{A counterexample}
\label{s:Lodh-counterex}

The following is an example of an $X/S$ where $X$ is \emph{regular}, but for which $K(\pi, 1)$ neighborhoods do not exist. We use the notation of \ref{setup}. 

\begin{proposition}\label{Lodh-counterex}
Suppose the characteristic of $k$ is 0. Let $\pi$ be a uniformizer of $V$ and let $X\subseteq \mathbb{A}^3_S$ be given by the equation $xy = z^2 - \pi$. Let $P = (0, 0, 0)\in X_k$ and let $U$ be an open neighborhood of $P$ in $X$. Then $U_\eta$ is not a $K(\pi, 1)$. 
\end{proposition}

The proof of Proposition \ref{Lodh-counterex} will be given at the end of this section.

\subsection{Remarks}
This is example not too surprising. Let us consider an analogous analytic family. It is easy to see that the Milnor fiber \cite{Milnor} at 0 of the function 
\[ f:\CC^3\to\CC, \quad f(x,y,z) = z^2 - xy, \]
(which is $x^2+y^2+z^2$ after a change of variables) is homotopy equivalent to the $2$-sphere. Moreover, its inclusion into $X\setminus f^{-1}(0)$ induces an isomorphism on $\pi_2$. It follows that for any open $U\subseteq \CC^3$ containing $0$, for some Milnor fiber the above isomorphism will factor through $\pi_2(U\setminus f^{-1}(0))$, therefore $U\setminus f^{-1}(0)$ cannot be $K(\pi, 1)$. The proof of Proposition \ref{Lodh-counterex} given below is an algebraic analog of this observation.

On the other hand, for $f:\CC^n\to \CC$ given by a monomial $x_1 \ldots x_r$, the Milnor fiber at 0 is homotopy equivalent to a torus $(S^1)^r$, which is a $K(\pi, 1)$. This explains why one should expect Theorems \ref{main} and \ref{vancycl} to be true.

\begin{lemma}\label{p1xp1}
There is an isomorphism $\PP^1_{K(\sqrt\pi)}\times \PP^1_{K(\sqrt\pi)}\setminus \text{(diagonal)} \to X_{K(\sqrt\pi)}$.
\end{lemma}

\begin{proof}
Let the coordinates on $\PP^1\times \PP^1$ be $((a:b),(c:d))$. The diagonal in $\PP^1\times\PP^1$ is defined by the equation $\delta:=ad-bc=0$. Let $f((a:b), (c:d)) = \sqrt{\pi}(2ac/\delta, 2bd/\delta, (bc+ad)/\delta)$. The inverse map is given by $g(x,y,z) = ((x:z-\sqrt{\pi}),(z-\sqrt\pi: y))$ on $\{z\neq \sqrt{\pi}\}$ and $g(x,y,z) = ((z+\sqrt\pi: y),(x:z+\sqrt{\pi}))$ on $\{z\neq -\sqrt{\pi}\}$. 
\end{proof}

\subsection{}
Let $C$ be the completed algebraic closure of $K$. Denote by $B$ the tube of $P$ in $X_{C}$:
\[ B = \{(x, y, z) \,:\, xy=z^2-\pi, |x|<1, |y|<1, |z|<1\} \subseteq X^{an}_{C}, \]
which we treat as a rigid analytic space over $C$.
Since $P\in U$, we have $B \subseteq U^{an}_{C}$. We denote by $p:B\to \PP^{1, an}_{C}$ the restriction to $B$ of the composition of $g:X^{an}_{C}\to \PP^{1,an}_{C}\times \PP^{1,an}_{C}$ defined in the proof of Lemma \ref{p1xp1} with the first projection $\PP^1_{C}\times \PP^1_{C}\to\PP^1_{C}$. Explicitly,
\[ p(x, y, z) = \begin{cases} (x: z-\sqrt{\pi}) & \text{if } z\neq\sqrt{\pi}, \\   
					          (z+\sqrt\pi: y) & \text{if } z\neq -\sqrt{\pi}. \\\end{cases} \]

Our goal is to prove that $p:B\to\PP^{1,an}_{C}$ is a fibration in open discs. If we were working with manifolds, we would deduce that $p$ is a homotopy equivalence which factors through $U$, hence $\PP^1$ is a homotopy retract of $U$, and deduce that $U$ is not a $K(\pi, 1)$ (in the usual sense of algebraic topology), because $\pi_2(U)$ contains $\pi_2(\PP^1)\isom\ZZ$ as a direct summand. In the rigid analytic setting, we do not have such tools at our disposal, but it is enough to show that $B$ is simply connected and that $p$ is injective on the second \'etale cohomology groups (see Corollary \ref{cor} below).

Let $U_+ = \{ (a:b)\in\PP^{1, an}_{C} \,:\, |a|\leq |b|\}$, $U_- = \{ (a:b)\in\PP^{1, an}_{C} \,:\, |a|\geq |b|\}$. This is an admissible cover of $\PP^{1, an}_{C}$. Let $B_+ = p^{-1}(U_+)$, $B_- = p^{-1}(U_-)$. Explicitly,
\[ B_+ = \{ (x, y, z) \,:\,  xy=z^2-\pi, \,|x|<1 ,\, |y|<1,\, |z|<1,\, |x|\leq|z-\sqrt{\pi}|,\, |z+\sqrt\pi| \leq |y|\}, \]
\[ B_- = \{ (x, y, z) \,:\,  xy=z^2-\pi, \, |x|<1 ,\, |y|<1,\, |z|<1,\, |x|\geq |z-\sqrt{\pi}|,\, |z+\sqrt\pi|\geq |y|\}. \]
This is an admissible cover of $B$. Let $U_0 = U_+\cap U_-$ and $B_0 = B_+ \cap B_- = p^{-1}(U_0)$.

\begin{lemma}
There are isomorphisms 
\[ h_+ : U_+ \times \{x \,:\,|x|<1\} \to B_+, \quad h_- : U_- \times \{y\,:\,|y|<1\}\to B_-, \]
\[ \text{and } h_0 : U_0\times \{z\,:\,|z|<1\} \to B_0, \]
 commuting with the projection $p$.
\end{lemma}

\begin{proof}
Let 
\begin{align*}
 h_+((a:b), x) &= \left(x,\, \frac{a}{b}\left(\frac{a}{b} x + 2\sqrt\pi\right),\, \frac{a}{b}x+\sqrt\pi\right), \\
 h_-((a:b), y) &= \left(\frac{b}{a}\left(\frac{b}{a} y - 2\sqrt\pi\right),\, y,\, \frac{b}{a}y-\sqrt\pi\right), \\
  h_0((a:b), z) &= \left(\frac{b}{a}\left(z-\sqrt\pi\right),\, \frac{a}{b}\left(z + \sqrt\pi\right),\, z\right). 
\end{align*}
Their inverses are given by $p\times x$, $p\times y$, $p\times z$, respectively. 
\end{proof}

\subsection{\'Etale topology of rigid analytic spaces} We will need the following facts (under the $\chara k = 0$ assumption):
\begin{itemize}
	\item the closed disc is simply connected \cite[6.3.2]{Berkovich},
	\item the open disc is simply connected (by the above and a limit argument),
	\item a product of (closed and open) discs has no nontrivial finite \'etale covers,
	\item the closed disc has trivial cohomology with $\ZZ/\ell$ coefficients (\cite[6.1.3]{Berkovich} combined with \cite[6.3.2]{Berkovich}),
	\item the open disc has trivial cohomology with $\ZZ/\ell$ coefficients (by \cite[6.3.12]{Berkovich} applied to the above),
	\item \'etale cohomology of rigid analytic spaces satisfies the K\"unneth formula.
\end{itemize}

\begin{corollary}\label{cor}
Suppose that $\chara k = 0$. Then $B$ is simply connected and the map
\[ p^*:H^2(\PP^{1,an}_{C, \et}, \ZZ/\ell)\to H^2(B_\et, \ZZ/\ell) \]
is an isomorphism for any prime number $\ell$.
\end{corollary}

\begin{proof}
Since $B_+$ and $B_-$ are both the product of an open disc and a closed disc and $\chara k = 0$, they are simply connected. Their intersection $B_0$ is connected, hence their union $B$ is simply connected (consider the sheaf of sections of a covering $B'\to B$).

Consider the two Mayer--Vietoris sequences for the \'etale cohomology of $\ZZ/\ell$ with respect to the two open covers $\{U_+, U_-\}$ and $\{B_+, B_-\}$. Since the latter cover is the pullback by $p$ of the former, $p$ induces a morphism of Mayer--Vietoris sequences
\[
	\xymatrix{
		\ldots\ar[r] & H^1(U_{+,\et}, \ZZ/\ell)\oplus H^1(U_{-,\et}, \ZZ/\ell) \ar[r]\ar[d]^{p^*} & H^1(U_{0,\et}, \ZZ/\ell)  \ar[r]\ar[d]^{p^*} & H^2(\PP^{1,an}_{C, \et}, \ZZ/\ell) \ar[d]^{p^*} \ar[r] & \ldots \\
		\ldots\ar[r] & H^1(B_{+,\et}, \ZZ/\ell)\oplus H^1(B_{-,\et}, \ZZ/\ell) \ar[r] & H^1(B_{0,\et}, \ZZ/\ell)  \ar[r] & H^2(B_\et, \ZZ/\ell) \ar[r] & \ldots
	}	
\]

Because $H^i(U_{\pm, \et}, \ZZ/\ell) = H^i(B_{\pm, \et}, \ZZ/\ell) = 0$ for $i=1,2$, the above diagram gives us a commutative diagram whose horizontal maps are isomorphisms
\[
	\xymatrix{
		H^1(U_{0,\et}, \ZZ/\ell)  \ar[r]^{\sim} \ar[d]^{p^*} & H^2(\PP^{1,an}_{C, \et}, \ZZ/\ell) \ar[d]^{p^*} \\
		H^1(B_{0,\et}, \ZZ/\ell)  \ar[r]^\sim & H^2(B_\et, \ZZ/\ell).
	}
\]
But $B_0$ is the product of $U_0$ with a disc, hence the left vertical map is an isomorphism. We conclude that the right map is an isomorphism as well.
\end{proof}

We can now prove Proposition \ref{Lodh-counterex}:

\begin{proof}[Proof of Proposition \ref{Lodh-counterex}.]
By Lemma \ref{kpi1prop}(c), it suffices to show that $U_C$ is not a $K(\pi, 1)$. Let $U' \to U_{C}$ be a finite \'etale morphism and let $B' \to B$ be the induced covering space of $B$. Pick a prime number $\ell$ and consider the diagram of pull-back maps
\[
  \xymatrix{
  	H^2(B'_\et, \ZZ/\ell)        &  H^2(U'^{an}_\et, \ZZ/\ell) \ar[l]  & \\
  	H^2(B_\et, \ZZ/\ell) \ar[u] &  H^2(U^{an}_\et, \ZZ/\ell)  \ar[u]\ar[l]  & H^2(\PP^{1, an}_{C,\et}, \ZZ/\ell). \ar[l]  
  }
\] 
Because $B$ is simply connected, $B'$ is a disjoint union of copies of $B$, hence $H^2(B_\et, \ZZ/\ell) \to H^2(B'_\et, \ZZ/\ell)$ is injective. Because $H^2(\PP^{1, an}_{C, \et}, \ZZ/\ell)\to H^2(B_\et, \ZZ/\ell)$ is an isomorphism, we conclude that $H^2(\PP^{1, an}_{C,\et}, \ZZ/\ell)\to H^2(U'^{an}_\et, \ZZ/\ell)$ is injective. As the spaces involved are smooth over $C$, by the rigid-\'etale comparison theorem \cite[Theorem 7.3.2]{deJongVDP}, $H^2(\PP^{1}_{C,\et}, \ZZ/\ell)\to H^2(U'_\et, \ZZ/\ell)$ is also injective. This shows that $U_C$ is not a $K(\pi, 1)$, because we cannot kill the image of $H^2(\PP^1_{C,\et}, \ZZ/\ell) \isom \ZZ/\ell$ in $H^2(U_{C,\et}, \ZZ/\ell)$ by a finite \'etale $U'\to U_C$.    
\end{proof}

\section{The equicharacteristic zero case}
\label{s:nakayama}

\begin{theorem}\label{vancycl}
Let $(X, \Mm_X)$ be a regular (cf. \cite[Definition 2.1]{KatoToric}) log scheme over $\QQ$ such that $X$ is locally noetherian, and let $\bar x\to X$ be a geometric point. Let $\Ff$ be a locally constant constructible abelian sheaf on $X^\circ:=(X, \Mm_X)_{tr}$, the biggest open subset on which $\Mm_X$ is trivial, and let $\zeta\in H^i(X^\circ, \Ff)$ for some $i>0$. There exists an \'etale neighborhood $U$ of $\bar x$ and a finite \'etale surjective map $V\to U^\circ$ such that $\zeta$ maps to zero in $H^i(V, \Ff)$.
\end{theorem}

\noindent \emph{Proof.}

\subsubsection{}
\label{secondred}
In proving the assertion, I claim that we can assume that $\Ff$ is constant. Let $Y\to X^\circ$ be a finite \'etale surjective map such that the pullback of $\Ff$ to $Y$ is constant. By the logarithmic version of Abhyankar's lemma \cite[Theorem 10.3.43]{GabberRamero}, $Y = X'^{\circ}$ for a finite and log \'etale $f:(X', \Mm_{X'})\to (X, \Mm_X)$. Then $(X', \Mm_{X'})$ is also log regular (by \cite[Theorem 8.2]{KatoToric}). Choose a geometric point $\bar x'\to X'$ mapping to $\bar x$, and let $\Ff' = f^{\circ *} \Ff$, which is a constant sheaf on $X'^{\circ}$. 

Suppose that we found an \'etale neighborhood $U'$ of $x'$ and a finite \'etale surjective map $V'\to U'^{\circ}$ killing $\zeta' := f^{\circ*} (\zeta) \in H^i(X'^{\circ}, \Ff')$. Let $X''$ be the normalization of $U'$ in $V'$, and choose a geometric point $\bar x''$ mapping to $\bar x'$. By \cite[Th\'eor\`eme 18.12.1]{EGA_IV_4} (or \stacks{02LK}), there exists a diagram 
\[
	\xymatrix{ 
		\bar{x}'' \ar[d] \ar[r]  & V \ar[d] \ar[r] & X'' \ar[d] \\
		\bar x \ar[r]         & U \ar[r]        & X
	}
\]
with $U\to X$ and $V\to X''$ \'etale and $V\to U$ finite. It follows that $V^\circ\to U^\circ$ is also \'etale, and that the pull-back of $\zeta$ to $V^\circ$ is zero.

In proving the theorem, we can therefore assume that $\Ff\isom \ZZ/n\ZZ$ for some integer $n$, by considering the direct summands.

\subsubsection{}
\label{existschart}
The question being \'etale local around $\bar x$, we can assume that there exists a chart $g:(X, \Mm_X)\to \AA_P$ for a fine saturated monoid $P$, which we use to form a cartesian diagram 
\begin{equation} \label{cartesianchart}
	\xymatrix{
	(X',\Mm_{X'}) \ar[r]^f \ar[d]_{g'}    & (X, \Mm_X) \ar[d]^g\\
	\AA_{P}  \ar[r]_{\AA_{\cdot n}}  & \AA_{P} .
	}	
\end{equation}
Then $(X', \Mm_{X'})$ is log regular, $f$ is finite, and $f:X'^{\circ} \to X^\circ$ is \'etale.
Choose a geometric point $\bar x'\to X'$ mapping to $\bar x$.  We have a commutative diagram
\[ 
	\xymatrix{
		\bar \Mm^{gp}_{X,\bar x} \ar[r]^{f^*} & {\bar \Mm}^{gp}_{X',\bar x'} \ \\
		\bar P \ar[u] \ar[r]_{\cdot n} & \bar P \ar[u]
	}
\]
where the vertical maps are surjections induced by the strict morphisms $g$ and $g'$. We conclude that the map
\begin{equation}\label{pullbackmap} 
f^\flat \otimes \ZZ/n\ZZ: \bar \Mm^{gp}_{X,\bar x}\otimes \ZZ/n\ZZ \to \bar \Mm^{gp}_{X'\bar x'} \otimes \ZZ/n\ZZ
\end{equation}
is zero.
 
\subsubsection{} 
Denote the inclusion $X^\circ\hookrightarrow X$ (resp. $X'^{\circ}\hookrightarrow X'$) by $j$ (resp. $j'$). By log absolute cohomological purity \eqref{puriso}, there is a functorial isomorphism
\[ R^i j_* (\ZZ/n\ZZ) \isom \bigwedge\nolimits^i (\bar \Mm^{gp}_X \otimes \ZZ/n\ZZ(-1)).  \]
In our situation, this means that there is a commutative diagram
\[\xymatrix{ 
	H^i(X'^{\circ}, \ZZ/n\ZZ) \ar[r]^{sp_{\bar x'}} & R^i j'_*(\ZZ/n\ZZ)_{\bar x'} \ar[r]^{\sim\hskip .2in} & \bigwedge\nolimits^i (\bar \Mm^{gp}_{X',\bar x'}\otimes \ZZ/n\ZZ) \\
	H^i(X^\circ, \ZZ/n\ZZ) \ar[r]_{sp_{\bar x}} \ar[u]^{f^*} & R^i j_* (\ZZ/n\ZZ)_{\bar x} \ar[r]^{\sim\hskip .2in} \ar[u]^{f^*} & \bigwedge\nolimits^i (\bar \Mm^{gp}_{X,\bar x}\otimes \ZZ/n\ZZ) \ar[u]_{\bigwedge^i \eqref{pullbackmap}}
}\]
where the rightmost map is zero for $i>0$ because \eqref{pullbackmap} is zero.

It follows that $\zeta$ maps to zero in $R^i j'_*(\ZZ/n\ZZ)_{\bar x'}$, hence there exists an \'etale neighborhood $U'$ of $\bar x'$ such that $\zeta$ maps to zero in $H^i(U'^{\circ}, \ZZ/n\ZZ)$. Applying once again the argument of the second paragraph of \ref{secondred} yields an \'etale neighborhood $U$ of $\bar x$ and a finite \'etale map $V\to U$ killing $\zeta$, as desired. \hfill \qed

\section{The comparison theorem}
\label{s:comparison}

\subsection{Faltings' site and Faltings' topos}
\label{ss:faltingstopos}

In \cite{AbbesGros}, Abbes and Gros have developed a theory of generalized co-vanishing topoi, of which the Faltings' topos is a special case. This topos has first been introduced in \cite{FaltingsAlmost}, though the definition of \cite{AbbesGros} is different. For reader's convenience, let us recall the definitions, adapting them to our setup. 

\begin{definition}\label{def:faltingstopos}
Let $f:Y\to X$ be a morphism of schemes. 
\begin{enumerate}[(a)]
	\item The \emph{Faltings' site} $E$ associated to $f$ is the site with
		\begin{itemize}
			\item {\sc objects} morphisms $V\to U$ over $f:Y\to X$ with $U\to X$ \'etale and $V\to U \times_X Y$ finite \'etale,
			\item {\sc morphisms} commutative squares over $f:Y \to X$,
			\item {\sc topology} generated by coverings of the following form:
				\begin{itemize}
					\item { (V, for vertical)} $\{ (V_i\to U) \to (V\to U)\}$ with $\{V_i\to V\}$ a covering,
					\item { (C, for cartesian)} $\{ (V\times_U U_i\to U_i) \to (V\to U) \}$ with $\{U_i\to U\}$ a covering. 
				\end{itemize}
		\end{itemize}
	\item The \emph{Faltings' topos} $\tilde{E}$ is the topos associated to $E$. 
	\item We denote by $\Psi:Y_\et\to \tilde E$ the morphism of topoi induced by the continuous map of sites $(V\to U)\mapsto V : E\to \Et_{/Y}$.
\end{enumerate}
\end{definition}

\begin{proposition}\label{someprop}
Let $f:Y\to X$ be a map of coherent schemes with the property that for every geometric point $\bar x$ of $X$, $X_{(\bar x)}\times_X Y$ has only a finite number of connected components, and let $\pp$ be a set of prime numbers. The following conditions are equivalent:
\begin{enumerate}[(a)]
	\item for every \'etale $U$ over $X$ and every $\pp$-torsion locally constant constructible abelian sheaf $\Ff$ on $U\times_X Y$, we have $R^i\Psi_{U*} \Ff = 0$ for $i>0$, where $\Psi_U : U\times_X Y \to \tilde E_{U}$ is the morphism \ref{def:faltingstopos}(c) for $U\times_X Y\to U$,
	\item for every \'etale $U$ over $X$, every $\pp$-torsion locally constant constructible abelian sheaf $\Ff$ on $U\times_X Y$, every class $\zeta\in H^i(U\times_X Y, \Ff)$ with $i>0$, and every geometric point $\bar x\to U$, there exists an \'etale neighborhood $U'$ of $\bar x$ in $U$ and a finite \'etale surjective map $V\to U\times_X Y$ such that the image of $\zeta$ in $H^i(V, \Ff)$ is zero.
	\item for every geometric point $\bar x\to X$, $X_{(\bar x)}\times_X Y$ is a $K(\pi, 1)$ for $\pp$-adic coefficients.
\end{enumerate}
\end{proposition}

\begin{proof}
The equivalence of (a) and (b) follows from the fact that $R^i \Psi_* \Ff$ is the sheaf associated to the presheaf $(V\to U)\mapsto H^i(V, \Ff)$ on $E$ and the argument in Proposition \ref{kpi1prop}(b). The equivalence of (b) and (c) is clear in the view of Proposition \ref{kpi1prop}.  
\end{proof}

\begin{corollary}\label{AtoB}
Suppose that $X$ has a basis for the \'etale topology consisting of $U$ for which $U\times_X Y$ is a $K(\pi, 1)$ for $\pp$-adic coefficients. Then the conditions of Proposition \ref{someprop} are satisfied. 
\end{corollary}

\begin{theorem} \label{comparison}
Let $(X, \Mm_X)$ be a log smooth log scheme over $(S, \Mm_S)$ such that $X_\eta$ is smooth over $\eta$, and let $X^\circ = (X, \Mm_X)_{tr}$. If $\chara k = 0$, assume moreover that $(X, \Mm_X)$ is saturated. Then for every geometric point $\bar x$ of $X$, $X_{(\bar x)}\times_X X^\circ_{\bar\eta}$ is a $K(\pi, 1)$.
\end{theorem}

\begin{proof}
We should first note that $X_{(\bar x)}\times_X X^\circ_{\bar\eta}$ is connected. Let $\bar S = \Spec \bar V$, where $\bar V$ is the integral closure of $V$ in $\bar K$. Then $(X_{(\bar x)})_{\bar S}$ is strictly local \cite[3.7]{Higgs2} and normal \cite[6.3(iii)]{Higgs1}, hence $X_{(\bar x)}\times_X X^\circ_{\bar\eta}$ is connected (being an open subset of the former).

In case $\chara k = 0$, as $(X, \Mm_X)$ is regular by \cite[Theorem 8.2]{KatoToric}, Theorem \ref{vancycl} implies condition (b) of Proposition \ref{someprop}, hence $X_{(\bar x)}\times_X X^\circ_{\eta}$ is a $K(\pi, 1)$ (note that $X^\circ\subseteq X_\eta$). As $X_{(\bar x)}\times_X X^\circ_{\bar\eta}$ is a limit of finite \'etale covers of $X_{(\bar x)}\times_X X^\circ_{\eta}$, it is a $K(\pi, 1)$ as well.

We will now assume that $\chara k = p > 0$ and follow \cite[Lemma 2.3]{Faltings} (see also \cite[4.9--4.5]{Olsson}). By Theorem \ref{main} and Corollary \ref{AtoB}, we know that $Z:=(X_{(\bar x)})_{\bar\eta}$ is a $K(\pi, 1)$. Since $X_\eta$ is smooth, $Z$ is regular and $Z^\circ = X_{(\bar x)}\times_X X^\circ_{\bar\eta}$ is obtained from $Z$ by removing divisor with strictly normal crossings $D = D_1\cup\ldots\cup D_r$. Let $\Ff$ be a locally constant constructible abelian sheaf on $Z^\circ$, and pick a $\zeta\in H^i(Z^\circ, \eta)$ ($i>0$). We want to construct a finite \'etale cover of $Z^\circ$ killing $\zeta$. 

By Abhyankar's lemma \cite[Exp. XIII, Appendice I, Proposition 5.2]{SGA1}, there is an integer $n$ such that if $f:Z'\to Z$ is a finite cover with ramification indices along the $D_i$ nonzero and divisible by $n$, then $f^{\circ *} \Ff$ extends to a locally constant constructible sheaf on $Z'$. I claim that we can choose $Z'$ which is a $K(\pi, 1)$. By the previous considerations, it suffices to find $Z'$ equal to $(X'_{(\bar x')})_{\bar\eta}$ for some $X'/S'$ satisfying the same assumptions as $X$. We can achieve this by choosing a chart $X\to \AA_{P}$ around $\bar x$ as before and taking a fiber product as in \eqref{cartesianchart} (and $S' = \Spec V[\pi']/(\pi'^{n} - \pi)$). 

Now that we can assume that $\Ff = j^* \Ff'$ where $ \Ff'$ is locally constant constructible on $Z$ and $j:Z^\circ \hookrightarrow Z$ is the inclusion, we choose a finite \'etale cover $g:Y\to Z$, Galois with group $G$, for which $g^* \Ff'$ is constant. 

Let $f:Z'\to Z$ be a finite cover with ramification indices along the $D_i$ nonzero and divisible by some integer $n$. I claim that for any $b\geq 0$, the base change map
\begin{equation}\label{divbyn}
	f^* R^b j_* \Ff \to R^b j'_{*} (f^{\circ *} \Ff) 
\end{equation}
is divisible by $n^b$. In case $\Ff$ is constant, this follows once again from logarithmic absolute cohomological purity \eqref{puriso}, and in general can be checked \'etale locally, e.g. after pulling back to $Y$, where $\Ff$ becomes constant. 
Consider the Leray spectral sequence for $j$:
\[ E^{a,b}_2 = H^a(Z, R^b j_* \Ff) \quad\Rightarrow\quad H^{a+b}(Z^\circ, \Ff),  \]
inducing an increasing filtration $F^b$ on $H^i(Z^\circ, \Ff)$. Let $b(\zeta)$ be the smallest $b\geq 0$ for which $\zeta\in F^b$. We prove the assertion by induction on $b(\zeta)$. If $b(\zeta) = 0$, then $\zeta$ is in the image of a $\zeta' \in H^i(Z, j_*\Ff)$, and since $j_* \Ff$ is locally constant and $Z$ is a $K(\pi, 1)$, we can kill $\zeta'$ by a finite \'etale cover of $Z$. For the induction step, let $n$ be an integer annihilating $\Ff$, and pick a ramified cover $f:Z'\to Z$ as in the previous paragraph, such that again $Z' = (X'_{(\bar x')})_{\bar\eta}$ for some $X'/S'$ satisfying the assumptions of the theorem. Note that since \eqref{divbyn} is divisible by $n$, it induces the zero map on $E^{a, b}_2$ for $b>0$, hence $b(f^*\zeta)<b(\zeta)$ and we conclude by induction.
\end{proof}

\begin{corollary} \label{comparisoncor}
Let $(X, \Mm_X)$ be as in Theorem \ref{comparison}, and let $X^\circ = (X, \Mm_X)_{tr}$. Consider the Faltings' topos $\tilde E$ of $X^\circ_{\bar\eta}\to X$ and the morphism of topoi
\[ \Psi : X^\circ_{\bar\eta,\et} \to \tilde E. \]
Let $\Ff$ be a locally constant constructible abelian sheaf on $X^\circ_{\bar\eta}$. Then $R^i\Psi_* \Ff = 0$ for $i>0$, and the natural maps \eqref{defmu}
\[ \mu: H^i(\tilde E, \Psi_*(\Ff)) \to H^i(X^\circ_{\bar\eta,\et}, \Ff) \]
are isomorphisms.
\end{corollary}

\bibliographystyle{plain} 
\bibliography{kpi1}

\end{document}